\documentclass[a4paper,12pt]{article}
\usepackage{amssymb,amsmath,amsthm,amsfonts}
\usepackage{latexsym}
\usepackage{graphicx}
\usepackage[pdftex,bookmarks,colorlinks=false]{hyperref}
\usepackage{verbatim,caption,subcaption}
\usepackage{enumitem}
\usepackage[hmargin=1.2in,vmargin=1.2in]{geometry}
\usepackage{authblk}
\usepackage{fancyhdr}
\newtheorem{theorem}{Theorem}[section]

\newtheorem{definition}[theorem]{Definition}

\newtheorem{problem}{Problem}
\newtheorem{proposition}[theorem]{Proposition}
\newtheorem{remark}[theorem]{Remark}

\title{\sc On Integer Additive Set-Sequential Graphs}

\author{\sc N K Sudev}
\affil{\small Department of Mathematics \\ Vidya Academy of Science \& Technology \\ Thalakkottukara, Thrissur - 680501, India.\\ email: {\em sudevnk@gmail.com}} 

\author{\sc K A Germina}
\affil{\small PG \& Research Department of Mathematics \\ Mary Matha Arts \&  Science College \\ Mnanthavady, Wayanad-670645, India. \\ email:{\em srgerminaka@gmail.com}}

\date{}

\begin{document}
\maketitle

\begin{abstract}
A set-labeling of a graph $G$ is an injective function $f:V(G)\to \mathcal{P}(X)$, where $X$ is a finite set of non-negative integers and a set-indexer of $G$ is  a set-labeling such that the induced function $f^{\oplus}:E(G)\rightarrow \mathcal{P}(X)-\{\emptyset\}$ defined by $f^{\oplus}(uv) = f(u){\oplus}f(v)$ for every $uv{\in} E(G)$ is also injective.  A set-indexer $f:V(G)\to \mathcal{P}(X)$  is called a set-sequential labeling of $G$ if $f^{\oplus}(V(G)\cup E(G))=\mathcal{P}(X)-\{\emptyset\}$. A graph $G$ which admits a set-sequential labeling is called a set-sequential graph. An integer additive set-labeling is an injective function $f:V(G)\rightarrow \mathcal{P}(\mathbb{N}_0)$, $\mathbb{N}_0$ is the set of all non-negative integers and an integer additive set-indexer is an integer additive set-labeling such that the induced function $f^+:E(G) \rightarrow \mathcal{P}(\mathbb{N}_0)$ defined by $f^+ (uv) = f(u)+ f(v)$ is also injective. In this paper, we extend the concepts of set-sequential labeling to integer additive set-labelings of graphs and provide some results on them.
\end{abstract}


\section{Introduction}

For all  terms and definitions, not defined specifically in this paper, we refer to \cite{FH} and for more about graph labeling, we refer to \cite{JAG}. Unless mentioned otherwise, all graphs considered here are simple, finite and have no isolated vertices. 

All sets mentioned in this paper are finite sets of non-negative integers. We denote the cardinality of a set $A$ by $|A|$. We denote, by $X$, the finite ground set of non-negative integers that is used for set-labeling the elements of $G$ and cardinality of $X$ by $n$. 

The research in graph labeling commenced with the introduction of $\beta$-valuations of graphs in \cite{AR}. Analogous to the number valuations of graphs, the concepts of set-labelings and set-indexers of graphs are introduced in \cite{A1} as follows.

Let $G$ be a $(p,q)$-graph. Let $X$, $Y$ and $Z$ be non-empty sets and $\mathcal{P}(X)$, $\mathcal{P}(Y)$ and $\mathcal{P}(Z)$ be their power sets. Then, the functions $f:V(G)\to \mathcal{P}(X)$, $f:E(G)\to \mathcal{P}(Y)$ and $f:V(G)\cup E(G)\to \mathcal{P}(Z)$ are called the {\em  set-assignments} of vertices, edges and elements of $G$ respectively. By a set-assignment of a graph, we mean any one of them.  A set-assignment is called a {\em set-labeling} or a {\em set-valuation} if it is injective. 

A graph with a set-labeling $f$ is denoted by $(G,f)$ and is referred to as a {\em set-labeled graph} or a {\em set-valued graph}. For a $(p,q)$- graph $G=(V,E)$ and a non-empty set $X$ of cardinality $n$, a {\em set-indexer} of $G$ is defined as an injective set-valued function $f:V(G) \rightarrow \mathcal{P}(X)$ such that the function $f^{\oplus}:E(G)\rightarrow \mathcal{P}(X)-\{\emptyset\}$ defined by $f^{\oplus}(uv) = f(u ){\oplus}f(v)$ for every $uv{\in} E(G)$ is also injective, where $\mathcal{P}(X)$ is the set of all subsets of $X$ and $\oplus$ is the symmetric difference of sets.

\begin{theorem}
\cite{A1} Every graph has a set-indexer.
\end{theorem}

Analogous to graceful labeling of graphs, the concept of set-graceful labeling and set-sequential labeling of a graph are defined in \cite{A1} as follows.

Let $G$ be a graph and let $X$ be a non-empty set. A set-indexer $f:V(G)\to \mathcal{P}(X)$  is called a {\em set-graceful labeling} of $G$ if $f^{\oplus}(E(G))=\mathcal{P}(X)-\{\emptyset\}$. A graph $G$ which admits a set-graceful labeling is called a {\em set-graceful graph}.

Let $G$ be a graph and let $X$ be a non-empty set. A set-indexer $f:V(G)\to \mathcal{P}(X)$  is called a {\em set-sequential labeling} of $G$ if $f^{\oplus}(V(G)\cup E(G))=\mathcal{P}(X)-\{\emptyset\}$. A graph $G$ which admits a set-sequential labeling is called a {\em set-sequential graph}.

\subsection{Integer Additive Set-Labeling of Graphs}

Let $A$ and $B$ be two non-empty sets. Then, their \textit{sum set}, denoted by $A+B$, is defined to be the set $A+B=\{a+b:a\in A, b\in B\}$. If $C=A+B$, then $A$ and $B$ are said to be the \textit{summands} of $C$. Using the concepts of sum sets of sets of non-negative integers, the notion of integer additive set-labeling of a given graph $G$ is introduced as follows.  

Let $\mathbb{N}_0$ be the set of all non-negative integers. An {\em integer additive set-labeling} (IASL, in short) of  graph $G$ is an injective function $f:V(G)\rightarrow \mathcal{P}(\mathbb{N}_0)$. A graph $G$ which admits an IASL is called an IASL graph. An {\em integer additive set-labeling} $f$ is an integer additive set-indexer (IASI, in short) if the induced function $f^+:E(G) \rightarrow \mathcal{P}(\mathbb{N}_0)$ defined by $f^+ (uv) = f(u)+ f(v)$ is injective.  A graph $G$ which admits an IASI is called an IASI graph.

The cardinality of the set-label of an element (vertex or edge) of a graph $G$ is called the {\em set-indexing number} of that element. An IASL (or an IASI) is said to be a $k$-uniform IASL (or $k$-uniform IASI) if $|f^+(e)|=k ~ \forall ~ e\in E(G)$. The vertex set $V(G)$ is called {\em $l$-uniformly set-indexed}, if all the vertices of $G$ have the set-indexing number $l$.

\begin{definition}{\rm
Let $G$ be a graph and let $X$ be a non-empty set. An integer additive set-indexer $f:V(G)\to \mathcal{P}(X)-\{\emptyset\}$  is called a {\em integer additive set-graceful labeling} (IASGL, in short) of $G$ if $f^{+}(E(G))=\mathcal{P}(X)-\{\emptyset,\{0\}\}$. A graph $G$ which admits an integer additive set-graceful labeling is called an {\em integer additive set-graceful graph} (in short, IASG-graph).}
\end{definition}

Motivated from the studies made in \cite{AGKS} and \cite{AH}, in this paper, we extend the concepts of set-sequential labelings of graphs to integer additive  set-sequential labelings and establish some results on them.

\section{IASSL of Graphs}

First, note that under an integer additive set-labeling, no element of a given graph can have $\emptyset$ as its set-labeling. Hence, we need to consider only non-empty subsets of $X$ for set-labeling the elements of $G$. 

Let $f$ be an integer additive set-indexer of a given graph $G$. Define a function $f^{\ast}:V(G)\cup E(G)\to \mathcal{P}(X)-\{\emptyset\}$ as follows.

\begin{equation}
f^{\ast}(x)= 
\begin{cases}
f(x) & \mbox{if} ~ x\in V(G)\\
f^+(x) & \mbox{if} ~ x\in E(G) 
\end{cases}
\label{eqn1}
\end{equation}

Clearly, $f^{\ast}[V(G)\cup E(G)]=f(V(G))\cup f^+(E(G))$. By the notation, $f^{\ast}(G)$, we mean $f^{\ast}[V(G)\cup E(G)]$. Then, $f^{\ast}$ is an extension of both $f$ and $f^+$ of $G$. Throughout our discussions in this paper, the function $f^{*}$ is as per the definition in Equation \eqref{eqn1}.

Using the definition of new induced function $f^{\ast}$ of $f$, we introduce the following notion as a sum set analogue of set-sequential graphs.

\begin{definition}{\rm
An IASI $f$ of $G$ is said to be an {\em integer additive set-sequential labeling} (IASSL) if the induced function $f^{\ast}(G)=f(V(G))\cup f^+{E(G))}=\mathcal{P}(X)-\{\emptyset\}$. A graph $G$ which admits an IASSL may be called an {\em integer additive set-sequential graph} (IASS-graph).}
\end{definition}

\noindent Hence, an integer additive set-sequential indexer can be defined as follows.

\begin{definition}\label{D-IASSG}{\rm
An integer additive set-sequential labeling $f$ of a given graph $G$ is said to be an \textit{integer additive set-sequential indexer} (IASSI) if the induced function $f^{\ast}$ is also injective. A graph $G$ which admits an IASSI may be called an {\em integer additive set-sequential indexed graph} (IASSI-graph).}
\end{definition}

A question that arouses much in this context is about the comparison between an IASGL and an IASSL of a given graph if they exist. The following theorem explains the relation between an IASGL  and an IASSL of a given graph $G$.

\begin{theorem}
Every integer additive set-graceful labeling of a graph $G$ is also an integer additive set-sequential labeling of $G$.
\end{theorem}
\begin{proof}
Let $f$ be an IASGL defined on a given graph $G$. Then, $\{0\}\in f(V(G))$ (see \cite{GS12}) and $|f^+(E(G))|=\mathcal{P}(X)-\{\emptyset,\{0\}\}$. Then, $f^{\ast}(G)$ contains all non-empty subsets of $X$. Therefore, $f$ is an IASSL of $G$.
\end{proof}

\noindent Let us now verify the injectivity of the function $f^{\ast}$ in the following proposition.

\begin{proposition}\label{P-ISSG1}{\rm
Let $G$ be a graph without isolated vertices. If the function $f^{\ast}$ is an injective, then no vertex of $G$ can have a set-label $\{0\}$.}
\end{proposition}
\begin{proof}
If possible let a vertex, say $v$, has the set-label $\{0\}$. Since $G$ is connected, $v$ is adjacent to at least one vertex in $G$. Let $u$ be an adjacent vertex of $v$ in $G$ and $u$ has a set-label $A\subset X$. Then, $f^{\ast}(u)=f(u)=A$ and $f^{\ast}(uv)=f^+(uv)=A$, which is a contradiction to the hypothesis that $f^{\ast}$ is injective.
\end{proof}

\noindent In view of Observation \ref{P-ISSG1}, we notice the following points.

\begin{remark}\label{R-IASSL0a}{\rm
Suppose that the function $f^{\ast}$ defined in \eqref{eqn1} is injective. Then, if one vertex $v$ of $G$ has the set label $\{0\}$, then $v$ is an isolated vertex of $G$.}
\end{remark}

\begin{remark}\label{R-IASSL0b}{\rm
If the function $f^{\ast}$ defined in \eqref{eqn1} is injective, then no edge of $G$ can also have the set label $\{0\}$.}
\end{remark}

The following result is an immediate consequence of the addition theorem on sets in set theory and provides a relation connecting the size and order of a given IASS-graph $G$ and the cardinality of its ground set $X$. 

\begin{proposition}\label{P-IASSL2a}
Let $G$ be a graph on $n$ vertices and $m$ edges. If $f$ is an IASSL of a graph $G$ with respect to a ground set $X$, then $m+n = 2^{|X|}-(1+\kappa)$, where $\kappa$ is the number of subsets of $X$ which is the set-label of both a vertex and an edge.
\end{proposition}
\begin{proof}
Let $f$ be an IASSL defined on a given graph $G$. Then, $|f^{\ast}(G)|=|f(V(G))\cup f^+(E(G))|=|\mathcal{P}(X)-\{\emptyset\}|=2^{|X|}-1$. But by addition theorem on sets, we have
\begin{eqnarray*}
|f^{\ast}(G)|& = & |f(V(G))\cup f^+(E(G))|\\
{\rm That~is},~~ 2^{|X|}-1 & = & |f(V(G))|+|f^+(E(G))|-|f(V(G))\cap f^+(E(G))|\\
& = & |V|+|E|- \kappa \\
\implies & = & m+n-\kappa\\
\therefore m+n & = & 2^{|X|}-1-\kappa.
\end{eqnarray*}
This completes the proof.
\end{proof}

We say that two sets $A$ and $B$ are of {\em same parity} if their cardinalities are simultaneously odd or simultaneously even. Then, the following theorem is on the parity of the vertex set and edge set of $G$.  

\begin{proposition}\label{T-IASSLG2}
Let $f$ be an IASSL of a given graph $G$, with respect to a ground set $X$. Then, if $V(G)$ and $E(G)$ are of same parity, then $\kappa$ is an odd integer and if $V(G)$ and $E(G)$ are of different parity, then $\kappa$ is an even integer, where $\kappa$ is the number of subsets of $X$ which are the set-labels of both vertices and edges.   
\end{proposition}
\begin{proof}
Let $f$ be a integer additive set-sequential labeling of a given graph $G$. Then, $f^{\ast}(G)= \mathcal{P}(X)-\{\emptyset\}$. Therefore, $|f^{\ast}(G)|= 2^{|X|}-1$, which is an odd integer.

{\em Case-1:} Let $V(G)$ and $E(G)$ are of same parity. Then, $|V|+|E|$ is an even integer. Then, by Proposition \ref{P-IASSL2a}, $2^{|X|}-1-\kappa$ is an even integer, which is possible only when $\kappa$ is an odd integer.

{\em Case-2:} Let $V(G)$ and $E(G)$ are of different parity. Then, $|V|+|E|$ is an odd integer. Then, by Proposition \ref{P-IASSL2a}, $2^{|X|}-1-\kappa$ is an odd integer, which is possible only when $\kappa$ is an even integer.
\end{proof}

A relation between integer additive set-graceful labeling and an integer additive set-sequential labeling of a graph is established in the following result.

\begin{theorem}
Every integer additive set-graceful labeling of a graph $G$ is also an integer additive set-sequential labeling of $G$.
\end{theorem}
\begin{proof}
Let $f$ be an IASGL defined on a given graph $G$. Then, $\{0\}\in f(V(G))$ and $|f^+(E(G))|=\mathcal{P}(X)-\{\emptyset,\{0\}\}$. Therefore, $\{0\}\in f^{\ast}(G)$. Then, $f^{\ast}(G)$ contains all non-empty subsets of $X$. Therefore, $f$ is an IASSL of $G$.
\end{proof}

The following result determines the minimum number of vertices in a graph that admits an IASSL with respect to a finite non-empty set $X$. 

\begin{theorem}\label{T-IASSG2}
Let $X$ be a non-empty finite set of non-negative integers. Then, a graph $G$ that admits an IASSL with respect to $X$ have at least $\rho$ vertices, where $\rho$ is the number of elements in $\mathcal{P}(X)$ which are not the sum sets of any two elements of $\mathcal{P}(X)$.
\end{theorem}
\begin{proof}
Let $f$ be an IASSL of a given graph $G$, with respect to a given ground set $X$. Let $\mathcal{A}$ be the collection of subsets of $X$ such that no element in $\mathcal{A}$ is the sum sets any two subsets of $X$. Since $f$ an IASL of $G$, all edge of $G$ must have the set-labels which are the sum sets of the set-labels of their end vertices. Hence, no element in $\mathcal{A}$ can be the set-label of any edge of $G$. But, since $f$ is an IASSL of $G$, $\mathcal{A}\subset f^{\ast}(G)=f(V(G))\cup f^+(E(G))$. Therefore, the minimum number of vertices of $G$ is equal to the number of elements in the set $\mathcal{A}$.
\end{proof}

The structural properties of graphs which admit IASSLs arouse much interests. In the example of IASS-graphs, given in Figure \ref{fig:G-IASSL1}, the graph $G$ has some pendant vertices. Hence, there arises following questions in this context. Do an IASS-graph necessarily have pendant vertices? If so, what is the number of pendant vertices required for a graph $G$ to admit an IASSL? Let us now proceed to find the solutions to these problems. 

The minimum number of pendant vertices required in a given IASS-graph is explained in the following Theorem.

\begin{theorem}\label{T-IASSL4}
Let $G$ admits an IASSL with respect to a ground set $X$ and let $\mathcal{B}$ be the collection of subsets of $X$ which are neither the sum sets of any two subsets of $X$ nor their sum sets are subsets of $X$. If $\mathcal{B}$ is non-empty, then
\begin{enumerate}[itemsep=0mm]
\item $\{0\}$ is the set-label of a vertex in $G$
\item the minimum number pendant vertices in $G$ is cardinality of $\mathcal{B}$.
\end{enumerate}

\begin{remark}
Since the ground set $X$ of an IASS-graph must contain the element $0$, every subset $A_i$ of $X$ sum set of $\{0\}$ and $A_i$ itself. In this sense, each subset $A_i$ may be considered as a \textit{trivial sum set} of two subsets of $X$.  
\end{remark}

In the following discussions, by a sum set of subsets of $X$, we mean the non-trivial sum sets of subsets of $X$.
 
\end{theorem}
\begin{proof}
Let $f$ be an IASSL of $G$ with respect to a ground set $X$. Also, let $\mathcal{B}$ be the collection of subsets of $X$ which are neither the sum sets of any two subsets of $X$ nor their sum sets are subsets of $X$. Let $A\subset X$ be an element of $\mathcal{B}$. then $A$ must be the set-label of a vertex of $G$. Since $A \in \mathcal{B}$, the only set that can be adjacent to $A$ is $\{0\}$. Therefore, since $G$ is a connected graph, $\{0\}$ must be the set-label of a vertex of $G$.
More over, since $A$ is an arbitrary vertex in $\mathcal{B}$, the minimum number of pendant vertices in $G$ is $|\mathcal{B}|$.  
\end{proof}

The following result thus establishes the existence of pendant vertices in an IASS-graph.

\begin{theorem}\label{T-IASSL3a}
Every graph that admits an IASSL, with respect to a non-empty finite ground set $X$, have at least one pendant vertex.
\end{theorem}
\begin{proof}
Let the graph $G$ admits an IASSL $f$ with respect to a ground set $X$. Let $\mathcal{B}$ be the collection of subsets of $X$ which are neither the sum sets of any two subsets of $X$ nor their sum sets are subsets of $X$. 

We claim that $\mathcal{B}$ is non-empty, which can be proved as follows. Since $X$ is a finite set of non-negative integers, $X$ has a smallest element, say $x_1$, and a greatest element $x_l$. Then, the subset $\{x_1,x_l\}$ belongs to $f^{\ast}(G)$. Since it is not the sum set any sets and is not a summand of any set in $\mathcal{P}(X)$, $\{x_1,x_l\} \in \mathcal{B}$. Therefore, $\mathcal{B}$ is non-empty.

Since $\mathcal{B}$ is non-empty, by Theorem \ref{T-IASSL3a}, $G$ has some pendant vertices.
\end{proof}

\begin{remark}
{\rm In view of the above results, we can make the following observations.
\begin{enumerate}[itemsep=0mm]
\item No cycle $C_n$ can have an IASSL.
\item For $n\ge 2$, no complete graph $K_n$ admits an IASSL.
\item No complete bipartite graph $K_{m,n}$ admits an IASL.
\end{enumerate}
} %
\end{remark}

The following result establish the existence of a graph that admits an IASSL with respect to a given ground set $X$.

\begin{theorem}\label{IASSL5a}
For any non-empty finite set $X$ of non-negative integers containing $0$, there exists a graph $G$ which admits an IASSL with respect to $X$. 
\end{theorem}
\begin{proof}
Let $X$ be a given non-empty finite set containing the element $0$ and let $\mathcal{A}=\{A_i\}$, be the collection of subsets of $X$ which are not the sum sets of any two subsets of $X$. Then, the set $\mathcal{A}'=\mathcal{P}(X)-\mathcal{A}\cup \{\emptyset\}$ is the set of all subsets of $X$ which are the sum sets of any two subsets of $X$ and hence the sum sets of two elements in $\mathcal{A}$. 

What We need is to construct a graph which admits an IASSL with respect to $X$. For this, begin with a vertex $v_1$. Label the vertex $v_1$ by the set $A_1=\{0\}$. For $1\le i \le |\mathcal{A_i}|$, create a new vertex $v_i$ corresponding to each element in $\mathcal{A}$ and label $v_i$ by the set $A_i\in \mathcal{A}$. Then, connect each of these vertices to $V_1$ as these vertices $v_i$ can be adjacent only to the vertex $v_1$. Now that all elements in $\mathcal{A}$ are the set-labels of vertices of $G$, it remains the elements of $\mathcal{A}'$  for labeling the elements of $G$. For any $A_r' \in \mathcal{A}'$, we have $A_r' = A_i+A_j$, where $A_i,A_j \in \mathcal{A}$. Then, draw an edge $e_r$ between $v_i$ and $v_j$ so that $e_r$ has the set-label $A_r'$. This process can be repeated until all the elements in $\mathcal{A}'$ are also used for labeling the elements of $G$. Then, the resultant graph is an IASS-graph with respect to the ground set $X$.
\end{proof}

\noindent Figure \ref{fig:G-IASSL1} illustrates the existence of an IASSL for a given graph $G$.

\begin{figure}[h!]
\centering
\includegraphics[width=0.7\linewidth]{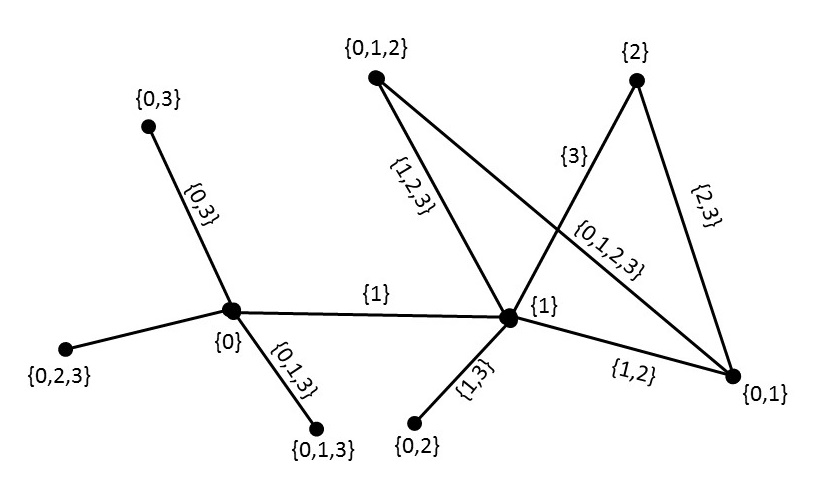}
\caption{}
\label{fig:G-IASSL1}
\end{figure}

On the other hand, for a given graph $G$, the choice of a ground set $X$ is also very important to have an integer additive set-sequential labeling. There are certain other restrictions in assigning set-labels to the elements of $G$. We explore the properties of a graph $G$ that admits an IASSL with respect to a given ground set $X$. As a result, we have the following observations. 

\begin{proposition}\label{P-ISSL3a}
Let $G$ be a connected integer additive set-sequential graph with respect to a ground set $X$. Let $x_1$ and $x_2$ be the two minimal non-zero elements of $X$. Then, no edges of $G$ can have the set-labels $\{x_1\}$ and $\{x_2\}$.
\end{proposition}
\begin{proof}
In any IASL-graph $G$, the set-label of an edge is the sum set of the set-labels of its end vertices. Therefore, a subset $A$ of the ground set $X$, that is not a sum set of any two subsets of $X$, can not be the set-label of any edge of $G$. Since $x_1$ and $x_2$ are the minimal non-zero elements of $X$, $\{x_1\}$ and $\{x_2\}$ can not be the set-labels of any edge of $G$. 
\end{proof}

\begin{proposition}\label{P-ISSL3b}
Let $G$ be a connected integer additive set-sequential graph with respect to a ground set $X$. Then, any subset $A$ of $X$ that contains the maximal element of $X$ can be the set-label of a vertex $v$ of $G$ if and only if $v$ is a pendant vertex that is adjacent to the vertex $u$ having the set-label $\{0\}$.
\end{proposition}
\begin{proof}
Let $x_n$ be the maximal element in $X$ and let $A$ be a subset of $X$ that contains the element $x_n$. If possible, let $A$ be the set-label of a vertex , say $v$, in $G$. Since $G$ is a connected graph, there exists at least one vertex in $G$ that is adjacent to $v$. Let $u$ be an adjacent vertex of $v$ in $G$ and let $B$ be its set-label. Then, the edge $uv$ has the set-label $A+B$. If $B \neq \{0\}$, then there exists at least one element $x_i\neq 0$ in $B$ and hence $x_i+x_n \not \in X$ and hence not in $A+B$, which is a contradiction to the fact that $G$ is an IASS-graph.
\end{proof}

Let us now discuss whether trees admit integer additive set-sequential labeling, with respect to a given ground set $X$.

\begin{theorem}
A tree $G$ admits an IASSL $f$ with respect to a finite ground set $X$, then $G$ has $2^{|X|-1}$ vertices.
\end{theorem}
\begin{proof}
Let $G$ be a tree on $n$ vertices. If possible, let $G$ admits an IASSI. Then, $|E(G)|=n-1$. Therefore, $|V(G)|+|E(G)|=n+n-1=2n-1$. But, by Theorem \ref{T-IASSG2}, $2^{|X|}-1=2n-1 \implies n=2^{|X|-1}$. 
\end{proof}

\noindent Invoking the above results, we arrive at the following conclusion.

\begin{theorem}
No connected graph $G$ admits an integer additive set-sequential indexer.
\end{theorem}
\begin{proof}
Let $G$ be a connected graph which admits an IASI $f$. By Proposition \ref{P-ISSG1}, if the induced function $f^{\ast}$ is injective, then $\{0\}$ can not be the set-label of any element of $G$. But, by Proposition \ref{P-ISSL3a} and Proposition \ref{P-ISSL3b}, every connected IASS-graph has a vertex with the set-label $\{0\}$. Hence, a connected graph $G$ can not have an IASSI.
\end{proof}

The problem of characterising (disconnected) graphs that admit IASSIs is relevant and interesting in this situation. Hence, we have

\begin{theorem}
A graph $G$ admits an integer additive set-sequential indexer $f$ with respect to a ground set $X$ if and only if $G$ has $\rho'$ isolated vertices, where $\rho'$ is the number of subsets of $X$ which are neither sum sets of any two subsets of $X$ nor the summands of any subsets of $X$.
\end{theorem}
\begin{proof}
Let $f$ be an IASI defined on $G$, with respect to a ground set $X$. Let $\mathcal{B}$ be the collection of subsets of $X$ which are neither sum sets of any two subsets of $X$ nor the summands of any subsets of $X$.

Assume that $f$ is an IASSI of $G$. Then, the induced function $f^{\ast}$ is an injective function.  We have already showed that $\mathcal{B}$ is a non-empty set. By Theorem \ref{T-IASSL4}, $\{0\}$ must be the set-label of one vertex $v$ in $G$ and the vertices of $G$ with set-labels from $\mathcal{B}$ can be adjacent only to the vertex $v$. By Remark \ref{R-IASSL0a}, $v$ must be an isolated vertex in $G$. Also note that $\{0\}$ is lso an element in $\mathcal{B}$. Therefore, all the vertices which have set-labels from $\mathcal{B}$ must also be isolated vertices of $G$. Hence $G$ has $\rho'=|\mathcal{B}|$ isolated vertices.

Conversely, assume that $G$ has $\rho'=|\mathcal{B}|$ isolated vertices. Then, label the  isolated vertices of $G$ by the sets in $\mathcal{B}$ in an injective manner. Now, label the other vertices of $G$ in an injective manner by other non-empty subsets of $X$ which are not the sum sets of subsets of  $X$ in such a way that the subsets of $X$ which are the sum sets of subsets of $X$ are the set-labels of the edges of $G$. Clearly, this labeling is an IASSI of $G$.
\end{proof}

Analogous to Theorem \ref{IASSL5a}, we can also establish the existence of an IASSI-graph with respect to a given non-empty ground set $X$.

\begin{theorem}
For any non-empty finite set $X$ of non-negative integers, there exists a graph $G$ which admits an IASSI with respect to $X$.
\end{theorem}

Figure \ref{fig:G-IASSI2} illustrates the existence of an IASSL for  a given graph with isolated vertices.

\begin{figure}[h!]
\centering
\includegraphics[width=0.7\linewidth]{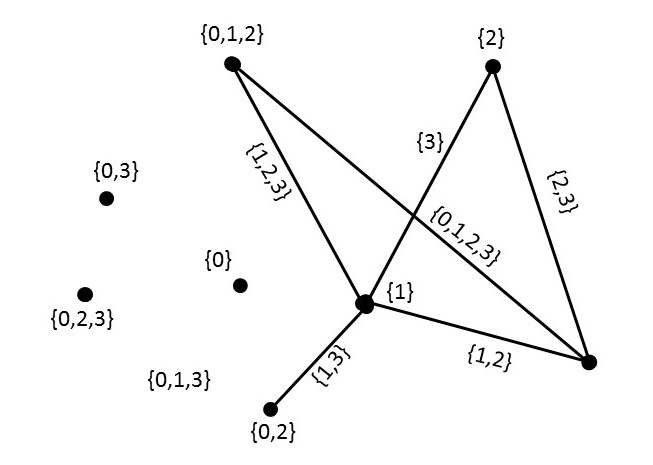}
\caption{}
\label{fig:G-IASSI2}
\end{figure}

\section{Conclusion}

In this paper, we have discussed an extension of set-sequential labeling of graphs to sum-set labelings and have studied the properties of certain graphs that admit IASSLs.  Certain problems regarding the complete characterisation of IASSI-graphs are still open.

We note that the admissibility of integer additive set-indexers by the graphs depends upon the nature of elements in $X$. A graph may admit an IASSL for some ground sets and may not admit an IASSL for some other ground sets. Hence, choosing a ground set is very important to discuss about IASSI-graphs.

Some of the areas which seem to be promising for further studies are listed below.

\begin{problem}{\rm
Characterise different graph classes which admit integer additive set-sequential labelings.}
\end{problem}

\begin{problem}{\rm
Verify the existence of integer additive set-sequential labelings for different graph operations, graph products and graph products.}
\end{problem}

The integer additive set-indexers under which the vertices of a given graph are labeled by different standard sequences of non negative integers, are also worth studying.   All these facts highlight a wide scope for further studies in this area.


\begin{thebibliography}{15}

\bibitem {A1} B D Acharya, {\bf Set-Valuations and Their Applications}, MRI Lecture notes in Applied Mathematics, The Mehta Research Institute of Mathematics and Mathematical Physics, Allahabad,1983.

\bibitem {A2} B D Acharya, {\em Set-Indexers of a Graph and Set-Graceful Graphs}, Bull. Allahabad Math. Soc., {\bf 16}(2001), 1-23.

\bibitem{AGPR} B D Acharya, K A Germina, K L Princy and S B Rao, (2008). {\em On Set-Valuations of Graphs}, in {\bf Labeling of Discrete Structures and Applications}, (Eds.: B D Acharya, S Arumugam and A Rosa), Narosa Publishing House, New Delhi.

\bibitem{AGKS} B D Acharya, K A Germina and Kumar Abhishek and P J Slater, (2012). {\em Some New Results on Set-Graceful and Set- Sequential Graphs}, Journal of Combinatorics, Information and System Sciences, {\bf 37}(2-4), 145-155.

\bibitem {AH} B D Acharya and S M Hegde, (1985). {\em Set-Sequential Graphs}, Nat. Acad. Sci. Letters, {\bf 8}(12)(1985), 387-390.

\bibitem {TMA} Tom M Apostol, (1989). {\bf Introduction to Analytic Number Theory}, Springer-Verlag, New York.
 
\bibitem {BM1} J A Bondy and U S R Murty, (2008). {\bf Graph Theory}, Springer.

\bibitem {BLS} A Brandst\"{a}dt, V B Le and J P Spinard, (1999). {\bf Graph Classes:A Survey}, SIAM, Philadelphia.

\bibitem {DMB} D M Burton, {\bf Elementary Number Theory}, Tata McGraw-Hill Inc., New Delhi, (2007).

\bibitem {CZ} G Chartrand and P Zhang, (2005). {\bf Introduction to Graph Theory}, McGraw-Hill Inc.

\bibitem {ND} N Deo, (1974). {\bf Graph Theory with Applications to Engineering and Computer Science}, PHI Learning.

\bibitem {JAG} J A Gallian, (2011). {\em A Dynamic Survey of Graph Labelling}, The Electronic Journal of Combinatorics (DS 16).

\bibitem{GK1} K A Germina and Kumar Abhishek, (2012). {\em Set-Valued Graphs - I}, ISPACS Journal of Fuzzy Set Valued Analysis, {\bf 2012}, Article IDjfsva-00127, 17 pages.

\bibitem {GA} K A Germina and T M K Anandavally, (2012). {\em Integer Additive Set-Indexers of a Graph:Sum Square Graphs}, Journal of Combinatorics, Information and System Sciences, {\bf 37}(2-4), 345-358.

\bibitem{GS1} K A Germina and N K Sudev, (2013). {\em On Weakly Uniform Integer Additive Set-Indexers of Graphs}, International Mathematical Forum, {\bf 8}(37), 1827-1834.
 
\bibitem {FH}  F Harary, (1969). {\bf Graph Theory}, Addison-Wesley Publishing Company Inc.

\bibitem{SMH} S M Hegde, (1991). {\em On Set-Valuations of Graphs}, Nat. Acad. Sci. Letters, Vol. 14(4), 181-182.

\bibitem{AR} A Rosa, (1967). {\em On certain valuation of the vertices of a graph}, In {\bf Theory of Graphs}, Gordon and Breach, New York and Dunod, Paris, [Proceedings of the International Symposium held in Rome].

\bibitem {GS0} N K Sudev and K A Germina, (2014). {\em On Integer Additive Set-Indexers of Graphs}, Int. J. Math. Sci.\& Engg. Applications, {\bf 8}(2), 11-22.

\bibitem {GS2} N K Sudev and K A Germina, (2015). {\em Some New Results on Strong Integer Additive Set-Indexers of Graphs}, Discrete Mathematics, Algorithms \& Applications, {\bf 7}(1), $11$ pages.

\bibitem {GS12} N K Sudev and K A Germina, (2014). {\em A Study on Integer Additive Set-Graceful Labelings of Graphs}, Submitted.

\bibitem {DBW} D B West, (2001). {\bf Introduction to Graph Theory}, Pearson Education Inc.

\end{thebibliography}
\end{document}